\documentclass[11pt]{amsart}
\usepackage{geometry}
\usepackage{amscd,amssymb,verbatim,xcolor}
\usepackage{longtable, multirow}
\usepackage{soul,cancel}

\date{\today}

\newcommand\bC{{\mathbb C}}

\newcommand\fa{{\mathfrak a}}

\newcommand\fg{{\mathfrak g}}
\newcommand\fh{{\mathfrak h}}
\newcommand\frk{{\mathfrak k}}

\newcommand\fp{{\mathfrak p}}
\newcommand\fq{{\mathfrak q}}

\newcommand\ft{{\mathfrak t}}

\newtheorem{theorem}{Theorem}[section]

\newtheorem{corollary}[theorem]{Corollary}
\newtheorem{conjecture}[theorem]{Conjecture}
\newtheorem{definition}[theorem]{Definition}
\newtheorem{example}[theorem]{Example}
\newtheorem{lemma}[theorem]{Lemma}
\newtheorem{proposition}[theorem]{Proposition}
\newtheorem{remark}[theorem]{Remark}

\newcommand\im{{\operatorname{im}}}

\begin{document}
\title[Scattered Representations]{Scattered representations of complex classical Lie groups}
\author{Chao-Ping Dong}
\author{Kayue Daniel Wong}

\address[Dong]{School of Mathematical Sciences, Soochow University, Suzhou 215006,
P.~R.~China}
\email{chaopindong@163.com}

\address[Wong]{School of Science and Engineering, The Chinese University of Hong Kong, Shenzhen,
Guangdong 518172, P. R. China}
\email{kayue.wong@gmail.com}

\begin{abstract}
As a continuation of \cite{DW}, this paper studies scattered representations
of $SO(2n+1, \bC)$, $Sp(2n, \bC)$ and $SO(2n, \bC)$. We describe the Zhelobenko parameters of these representations, count their cardinality, and determine their spin-lowest $K$-types. We also disprove a conjecture raised in 2015 asserting that the unitary dual can be obtained via parabolic induction from irreducible unitary representations with non-zero Dirac cohomology.
\end{abstract}



\maketitle
\setcounter{tocdepth}{1}

%%%%%%%%%%%%%%%%%%%
%%%%%%%%%%%%%%%%%%%
%%%SECTION 1%%%%%%%
%%%%%%%%%%%%%%%%%%%
%%%%%%%%%%%%%%%%%%%
\section{Introduction}\label{sec:intro}

Although many results quoted in this paper hold in a much wider setting, for convenience, we set $G$ as a complex connected classical Lie group. Fix a Cartan involution $\theta$ of $G$ such that its fixed points form a maximal compact subgroup $K$ of $G$. Then on the Lie algebra level, we have the Cartan decomposition
$$
\fg_0=\frk_0+\fp_0.
$$
The subscripts will be dropped to stand for the complexified Lie algebras

In the late 1990s, Vogan introduced the notion of Dirac cohomology \cite{V}, see \eqref{Dirac-cohomology}. As a new invariant for admissible $(\mathfrak{g}, K)$-modules, Dirac cohomology has profound implications in various areas of representation theory and beyond.

One of the applications of Dirac cohomology is to gain a better understanding of unitary representations.
More explicitly, it is known by \cite{VZ} that all unitary $(\mathfrak{g},K)$-modules with non-zero $(\mathfrak{g},K)$ cohomology are $A_\fq(\lambda)$ modules. Moreover, in \cite{S}, it was shown that these modules characterize all unitary modules with real, integral, and strongly regular infinitesimal characters. On the other hand,
it is shown in \cite{HKP} that the set of unitary modules with non-zero Dirac cohomology
$\widehat{G}^d$ strictly contains all unitary modules with non-zero $(\mathfrak{g},K)$ cohomology.
Therefore, by understanding $\widehat{G}^d$, one can have a better understanding of the unitary dual $\widehat{G}$.

The paper \cite{DD} reduced the study of $\widehat{G}^d$ to a
finite set called scattered representations $\widehat{G}^{sc}$, see Definition \ref{def-sc} for the precise meaning. Let us recall some of the recent progress in the
study of $\widehat{G}^{sc}$ in the case when $G$ is a complex simple Lie group:
\begin{itemize}
\item When $G = SL(n,\mathbb{C})$, \cite{DW} gives a complete study of $\widehat{G}^{sc}$, thus of $\widehat{G}^{d}$ as well.
\item When $G$ is  exceptional of type $G_2$, $F_4$ and $E_6$, then
\cite{DD} and \cite{D2} described the Zhelobenko parameters of all scattered representations.
One can then use \texttt{atlas} \cite{At} to verify that these modules satisfy Conjecture 1.1 of \cite{BP1}.
\end{itemize}

In this manuscript, we study scattered representations of $G$ other than type A. Our main results include Corollary \ref{Cor-usmall} answering \cite[Conjecture C]{DD} for $G$ in the affirmative, Corollary \ref{Cor-number} on the number of scattered representations, and Theorem \ref{thm-slkt} giving the form of spin-lowest $K$-types introduced in \cite{D1} (see Section \ref{sec-slkt}). At the end of the introduction, we will prove Conjecture 5.6 (ii)-(iv) of \cite{DD}. In Section \ref{sec-Hconj}, we will disprove Conjecture 13.3 of \cite{H}, which asserts that any irreducible unitary representation can be parabolically induced from an irreducible unitary representation with non-zero Dirac cohomology. As we shall see, this conjecture captures some special feature of type $A$, which does not hold in general.

There is a couple of reasons why we are interested in studying $\widehat{G}^{sc}$ for these groups
even though $\widehat{G}^d$ is already known. Firstly, when $\pi \in \widehat{G}^d$, spin lowest $K$-types are precisely the $K$-types contributing to the Dirac cohomology of
$\pi$ (see Section \ref{sec-slkt} below). The main result of \cite{BDW} gives the Dirac cohomology of all $\pi \in \widehat{G}^d$
for all complex classical groups, yet their spin lowest $K$-types are unknown. Therefore, it
would still be advantageous to investigate the spin lowest $K$-types of $\widehat{G}^{sc}$
even though $\widehat{G}^d$ is known in such cases.

Secondly, the recent research announcement \cite{BP2} suggests that Dirac cohomology could be used to construct interesting automorphic forms. The techniques involved in the study of $\widehat{G}^{sc}$ in this manuscript should be applicable to
(some) real reductive groups.

\subsection{Preliminaries}
Let $H$ be a Cartan subgroup of $G$. Let $\fh_0$ be the Lie algebra of $H$. Fix a positive root system $\Delta_G^+$ of $\Delta(\fg_0, \fh_0)$, and let $\rho$ be half the sum of all positive roots in $\Delta_G^+$.

We denote by $\langle\cdot, \cdot\rangle$ the Killing form on $\fg_0$. This form is negative definite on $\frk_0$ and positive definite on $\fp_0$. Moreover, $\frk_0$ and $\fp_0$ are orthogonal to each other under $\langle\cdot, \cdot\rangle$. We shall denote by $\|\cdot\|$ the norm corresponding to the Killing form. Let $Z_i$, $1\leq i\leq m$, be an orthonormal basis of $\fp_0$ with respect to the norm $\|\cdot\|$. The \emph{Dirac operator} $$
D:=\sum_{i=1}^m Z_i\otimes Z_i
$$
introduced by Parthasarathy \cite{P1} lives in $U(\fg)\otimes C(\fp)$, the tensor product of the universal enveloping algebra of $\fg$ and the Clifford algebra of $\fp$. Denote by
${\rm Ad}: K\to {\rm SO}(\fp_0)$ the adjoint map, and $p: {\rm Spin} (\fp_0)\to {\rm SO}(\fp_0)$ the double covering map.
Let $\widetilde{K}$ be the spin double covering group of $K$. That is,
$$
\widetilde{K}:=\{(k, s) \in K \times {\rm Spin} (\fp_0) \mid {\rm Ad}(k)=p(s) \}.
$$
Let $S$ be a spin module of $C(\fp)$, and let $\pi$ be an admissible $(\fg, K)$-module. Then $D$ acts on $\pi\otimes S$, and the \emph{Dirac cohomology} of $\pi$ \cite{V} is defined as
\begin{equation}\label{Dirac-cohomology}
H_D(\pi):=\ker D/(\ker D \cap \im D).
\end{equation}
Since the Dirac operator $D$ is independent of the choice of the orthonormal basis $\{Z_i\}_{i=1}^m$, it is evident that as a $\widetilde{K}$-module, the Dirac cohomology $H_D(\pi)$ is an invariant of $\pi$.

Now let us come to the representations of $G$. Let $H=TA$ be the Cartan decomposition of $H$, with $\fh_0=\ft_0+\fa_0$. We make the following identifications:
\begin{equation}\label{identification}
\fh\cong \fh_0\times \fh_0, \quad \ft=\{(x, -x): x\in\fh_0\} \cong \fh_0, \quad \fa\cong\{(x, x): x\in \fh_0\}  \cong \fh_0.
\end{equation}
Take an arbitrary pair $(\lambda_L, \lambda_R)\in \fh_0^*\times \fh_0^*$ such that $\eta:=\lambda_L-\lambda_R$ is integral.
Define
\begin{equation}\label{def-eta}
\{\eta\} := w'\eta
\end{equation}
as the unique dominant weight to which $\eta$ is conjugate under the action of the Weyl group $w' \in W := W(\fg_0,\fh_0)$.
Write $\nu:=\lambda_L + \lambda_R$. We can view $\eta$ as a weight of $T$ and $\nu$ a character of $A$. Put
$$
I(\lambda_L, \lambda_R):={\rm Ind}_B^G(\bC_{\eta}\otimes \bC_{\nu}\otimes {\rm triv})_{K-{\rm finite}},
$$
where $B\supset H$ is the Borel subgroup of $G$ determined by $\Delta_G^+$.
Then the $K$-type with highest weight $\{\eta\}$, denoted by $V^K_{\{\eta\}}$, occurs exactly once in $I(\lambda_L, \lambda_R)$. Let $J(\lambda_L, \lambda_R)$ be the unique irreducible subquotient of $I(\lambda_L, \lambda_R)$ containing $V^K_{\{\eta\}}$. By \cite{Zh}, every irreducible admissible $(\fg, K)$-module has the form $J(\lambda_L, \lambda_R)$. Indeed, $J(\lambda_L, \lambda_R)$ has infinitesimal character equal to the $W(\fg,\fh) = W \times W$ orbit of $(\lambda_L, \lambda_R)$, and lowest $K$-type $V^K_{\{\lambda_L-\lambda_R\}}$ (note that under the identification $\ft \cong \fh_0$ above, the compact Weyl group $W(\mathfrak{k},\ft)$ is isomorphic to $W$). We will refer to the pair $(\lambda_L, \lambda_R)$ as the {\it Zhelobenko parameter} for the module $J(\lambda_L, \lambda_R)$.

 For $J(\lambda_L, \lambda_R)$ to live in $\widehat{G}^d$, it should admit a nondegenerate Hermitian form in the first place. Moreover, it should satisfy the  Vogan conjecture proved by Huang and Pand\v{z}i\'{c} \cite{HP1}. Then one deduces that, as carried out on page 5 of \cite{BP1}, $J(\lambda_L, \lambda_R)$ must have the form
 \begin{equation}\label{BP-reduction}
 J(\lambda, -s\lambda), \quad 2\lambda \mbox{ is dominant integral regular and}\ s\in W \mbox{ is an involution}.
 \end{equation}

\subsection{Spin-lowest $K$-type} \label{sec-slkt}
From now on, let $\pi = J(\lambda,-s\lambda)$ be an irreducible unitary $(\fg, K)$-module as in \eqref{BP-reduction}.
To achieve the classification of $\widehat{G}^d$, the first-named author
introduced the notion of spin-lowest $K$-type \cite{D1}: Given an arbitrary $K$-type $V^K_{\delta}$,
its \textit{spin norm} is defined as
\begin{equation}\label{spin-norm}
\|\delta\|_{\rm spin}:=\|\{\delta-\rho\}+\rho\|
\end{equation}
(here we use notation \eqref{def-eta} for $\{\delta-\rho\}$). Then a $K$-type $V^K_{\tau}$ occurring in $\pi$ is called a \textit{spin-lowest $K$-type} of $\pi$ if it attains the minimum spin norm among all the $K$-types showing up in $\pi$.

Using the definition of spin norm, \textit{Parthasarathy's Dirac operator inequality} \cite{P2} can be rephrased as:
\begin{equation}\label{Dirac-inequality}
\|\delta\|_{\rm spin}\geq \|2\lambda\| \quad \text{for all}\ V^K_{\delta}\ \text{appearing in}\ \pi.
\end{equation}
Moreover, one can deduce from \cite[Theorem 3.5.3]{HP2} that $\pi \in \widehat{G}^d$ if and only if
\begin{equation} \label{eq-par}
\{\mu-\rho\}+\rho = 2\lambda
\end{equation}
for some $K$-type $V^K_{\mu}$ appearing in $\pi$. In other words,
$\pi \in \widehat{G}^d$ if and only if its spin lowest $K$-types
make \eqref{Dirac-inequality} an equality. Moreover, in such cases, the Dirac cohomology of $\pi$ consists only of certain copies of
the $\widetilde{K}$-module $V^{\widetilde{K}}_{2\lambda - \rho}$.
Put in a different way, if $\pi \in \widehat{G}^d$,
then the spin-lowest $K$-types of $\pi$ are exactly the $K$-types contributing to its Dirac cohomology
(see Proposition 3.3 of \cite{D1} for more details).

\subsection{The classification of $\widehat{G}^d$}

Let us recall the classification of $\widehat{G}^d$ in \cite{BDW}.

\begin{theorem} \label{thm-unitarydual}
Let $G$ be a connected complex classical  Lie group, and $\pi = J(\lambda,-s\lambda)$
be an irreducible $(\mathfrak{g}, K)$-module as in \eqref{BP-reduction}. Then $\pi \in \widehat{G}^d$ if and only if  it has the form
$$\pi = {\rm Ind}_L^G\left(\bigotimes_{i=1}^l \mathbb{C}_{\chi_i} \otimes \pi_{\mathcal{O}}\right),$$
where
\begin{itemize}
\item $L$ is the Levi subgroup with $L = \prod_{i=1}^l GL(a_i) \times G'$, with $G'$ being the same type as $G$ of
smaller rank (we also allow $G' = G$, or $G'$ to be absent in $L$);
\item $\mathbb{C}_{\chi_i}$ is a unitary character of $GL(a_i)$; and
\item $\pi_{\mathcal{O}}$ is equal to one of the following unipotent representations of $G'$:
\begin{align*}
&\text{Type}\ B:\ \pi_{[2k+1,2n]} := J\begin{pmatrix} \frac{1}{2}, \frac{3}{2}, \dots, \frac{2k-1}{2}; 1, 2, \dots, n \\
\frac{1}{2}, \frac{3}{2}, \dots, \frac{2k-1}{2}; 1, 2, \dots, n \end{pmatrix}, \quad k \geq n \geq 0 \\
&\text{Type}\ C: \begin{cases} \pi_{[2n]} := J\begin{pmatrix}  1, 2, \dots, n \\
1, 2, \dots, n \end{pmatrix},\quad  n > 0\\
\pi_{[2n-1,1]} := J\begin{pmatrix}
&\frac{1}{2}, &\frac{3}{2}, &\dots, &\frac{2n-1}{2} \\
&\frac{(-1)^n}{2}, &\frac{3}{2}, &\dots, &\frac{2n-1}{2} \end{pmatrix} ,\quad  n > 0
\end{cases} \\
&\text{Type}\ D: \begin{cases}
\pi_{[2n]} := J\begin{pmatrix}0,\dots
  ,n-1 \\ 0,\dots ,n-1 \end{pmatrix},\quad  n > 0
 \\
\pi_{[2n,2k-1,1]} := J\begin{pmatrix}0,\dots
  ,n-1, &\frac{1}{2}, &\dots, & k - \frac{1}{2} \\ 0,\dots
  ,n-1, & \frac{(-1)^k}{2}, &\dots , &k- \frac{1}{2}
  \end{pmatrix},\quad  n \geq k > 1
	\end{cases}
\end{align*}
where $J\begin{pmatrix} \lambda_L \\ \lambda_R \end{pmatrix} := J(\lambda_L,\lambda_R)$ in the above expressions, 
and the subscripts of above representations give the column sizes of the Young diagram of the nilpotent orbit $\mathcal{O} \subset \mathfrak{g}'$.
\end{itemize}
\end{theorem}

\begin{remark}
\emph{(a)} \ We use the notation $\mathrm{Ind}_L^G(\pi_L)$ to denote the (real) parabolic induction $\mathrm{Ind}_{LN}^G(\pi_L \otimes \mathrm{triv})$ for any choice of the unipotent radical $N$. The induced representations are indeed isomorphic for different choices of $N$.

\emph{(b)} \ In Type D of the above theorem, we excluded the unipotent representations $\pi_{[2n, 1, 1]} = J
\begin{pmatrix}
&0, &\dots, & n-1, &\frac{1}{2}\\
&0, &\dots, & n-1, &-\frac{1}{2}
\end{pmatrix}$ which appear in \cite{BP1}. In fact,
$\pi_{[2n, 1, 1]} ={\rm Ind}_L^G(\mathbb{C}_{(1)} \otimes \pi_{[2n]})$, where $L$ is the Levi subgroup $L = GL(1) \times SO(2n)$.\\

\emph{(c)}\ Note that Conjecture 5.6 (ii)-(iv) of \cite{DD} follows directly from Theorem \ref{thm-unitarydual}. More precisely, any spherical, non-trivial representation in Theorem \ref{thm-unitarydual} must be a unipotent representation $\pi_{\mathcal{O}}$ with Zhelobenko parameters $J\begin{pmatrix} \lambda \\ \lambda \end{pmatrix}$. This includes exactly the following representations:
\begin{itemize}
\item[Type $B$:] $\pi_{[2k+1,2n]}$ for $k \geq n \geq 0$, with $2\lambda = (\overbrace{2k-1,2k-3, \dots, 2n+1}^{k-n}, 2n, \dots, 2,1)$. In terms of fundamental weights, $2\lambda$ is equal to
 $$[\overbrace{2,\dots,2}^{a},\overbrace{1,\dots,1}^{k+n-a-1},2]$$
where $a := \max\{k-n-1,0\}$.
\item[Type $C$:] The spherical metaplectic (or minimal) representation $\pi_{[2n-1,1]}$ {\bf for $n$ even}, with $2\lambda = (2n-1, 2n-3, \dots, 1)$. In terms of fundamental weights, $2\lambda$ is
equal to
$$[\overbrace{2,\dots,2}^{n-1},1].$$
\item[Type $D$:] $\pi_{[2n,4r-1,1]}$ with $n \geq 2r > 1$, with $2\lambda = (\overbrace{2n-2, 2n-4, \dots, 4r}^{n-2r}, 4r-1, \dots, 1,0)$.
In terms of fundamental weights, $2\lambda$ is equal to
$$[\overbrace{2,\dots,2}^{a},\overbrace{1,\dots,1}^{n+2r-a}],$$
where $a := \max\{n-2r-1,0\}$.
\end{itemize}
\end{remark}

\section{Scattered Representations}
In this section, we investigate which representations in Theorem \ref{thm-unitarydual} are scattered. Firstly, we recall the definition of scattered representations.

\begin{definition}\label{def-sc}
Let $\pi = J(\lambda,-s\lambda) \in \widehat{G}^d$. Then $\pi$ is a {\bf scattered representation} of $G$ if all the simple reflections of $W$
occur in any reduced expression of $s$. In such a case, we write that $\pi\in \widehat{G}^{sc}$.
\end{definition}

By Theorem A of \cite{DD},  $\widehat{G}^{sc}$ is a finite set. Moreover, the set $\widehat{G}^d$ can be completely recovered from the scattered representations  of $[L, L]$, where $L$ runs over the Levi factors of all the $\theta$-stable parabolic subgroups of $G$.

As in the case of $GL(n,\mathbb{C})$ in \cite{DW}, we use chains to
study the representations in $\widehat{G}^{sc}$.

\begin{definition}
The chains of each classical type is defined as follows:

\noindent \emph{(i)} Type A: Let $T \geq t$ be positive integers of the same parity, define
$$\mathcal{A}_{T,t} := \{T, T-2, \dots, t+2, t\}_A := \begin{pmatrix}
\frac{T}{2} & \frac{T-2}{2}& \dots & \frac{t+2}{2} & \frac{t}{2} \\
-\frac{t}{2} & -\frac{t+2}{2} & \dots & -\frac{T-2}{2} & -\frac{T}{2} \\
\end{pmatrix}$$
\noindent \emph{(ii)} Type B: Let $k \geq n \geq 0$ be integers, define
$$\mathcal{B}_{[2k+1,2n]} := \{2k-1, 2k-3, 2n+1, 2n, 2n-1, \dots, 1\}_B := \begin{pmatrix} \frac{1}{2}, \frac{3}{2}, \dots, \frac{2k-1}{2}; 1, 2, \dots, n \\
\frac{1}{2}, \frac{3}{2}, \dots, \frac{2k-1}{2}; 1, 2, \dots, n \end{pmatrix} $$
\noindent \emph{(iii)} Type C: Let $n$ be a positive integer, define
\begin{align*}
\mathcal{C}_{[2n]} &:= \{2n,2n-2,\dots,2\}_C := \begin{pmatrix}{1,\dots
  ,n} \\ {1,\dots
  ,n} \end{pmatrix}; \\
\mathcal{C}_{[2n-1,1]} &:= \{2n-1,2n-3,\dots,1\}_C := \begin{pmatrix}{ \frac{1}{2},\dots
  ,\frac{2n-1}{2}} \\ { \frac{(-1)^n}{2},\dots
  ,\frac{2n-1}{2}} \end{pmatrix}
\end{align*}
\noindent \emph{(iv)} Type D: Let $n > 0$ be an integer, define
$$\mathcal{D}_{[2n]} := \{2n-2,2n-4,\dots,0\}_D := \begin{pmatrix}{0,\dots,
  n-1} \\ {0,\dots,
  n-1} \end{pmatrix};$$
Also, let $n \geq k > 1$ be integers, define	
\begin{align*}
\mathcal{D}_{[2n,2k-1,1]} &:= \{2n-2,2n-4,\dots,2k,2k-1,\dots,1,0\}_D \\
&:= \begin{pmatrix} 0,\dots
  ,n-1; &\frac{1}{2},\dots ,\frac{2k-1}{2} \\ 0,\dots
  ,n-1; &\frac{(-1)^k}{2},\dots ,\frac{2k-1}{2} \end{pmatrix}
	\end{align*}

If the subscript of a chain is omitted, we assume that the chain is of Type $A$.
\end{definition}

Using Theorem \ref{thm-unitarydual} and the definition above, all $J(\lambda,-s\lambda) \in \widehat{G}^d$
have Zhelobenko parameters of the form
$$(\lambda, -s\lambda) = \mathcal{X}_{\mathcal{O}}   \quad \text{or} \quad   \bigcup_{i=1}^l \mathcal{A}_{T_i, t_i} \quad \text{or} \quad \bigcup_{i=1}^l \mathcal{A}_{T_i, t_i} \cup \mathcal{X}_{\mathcal{O}},$$
where $\mathcal{X} =$ $\mathcal{B}$, $\mathcal{C}$ or $\mathcal{D}$, and the coordinates of all chains constitute $2\lambda$. In particular, the coordinates of the union of chains must be distinct.

In order to show which $(\lambda, -s\lambda)$ gives rise to scattered representations, we
extend the definition of interlaced chains in \cite{DW} as follows:
\begin{definition}
\begin{itemize}
\item[(a)] Two chains $\mathcal{X}_1 = \{M, \dots, m\}_X$, $\mathcal{X}_2 = \{N, \dots, n\}_X$,
where $X = A, B, C$ or $D$, are {\bf linked} if
the entries of $\mathcal{X}_1$ and $\mathcal{X}_2$ are disjoint, and either one of the following holds:
\begin{itemize}
\item[$\bullet$] $M > N > m $; or
\item[$\bullet$] $N > M > n$; or
\item[$\bullet$] $\{\mathcal{X}_1$, $\mathcal{X}_2\}$ $=$ $\{\mathcal{C}_{[2n]}, \mathcal{A}_{(1,1)}\}$.
\end{itemize}
\item[(b)] We say a union of chains  $\displaystyle \bigcup_{i \in I} \mathcal{X}_i$ is {\bf interlaced} if for all $i \neq j$ in $I$,
there exist indices $i = i_0, i_1, \dots, i_k = j$ in $I$ such that $\mathcal{X}_{i_{l-1}}$ and
$\mathcal{X}_{i_{l}}$ are linked for all $1 \leq l \leq k$.
(By convention, we also consider the single chain $\mathcal{X}$ to be interlaced).
\end{itemize}
\end{definition}

Among all $\pi \in \widehat{G}^d$ given in \cite{BDW},
the following theorem allows us to pick out those appearing in $\widehat{G}^{sc}$.

\begin{theorem} \label{thm-scattered}
Take any $\pi\in\widehat{G}^d$ as described in Theorem \ref{thm-unitarydual}.
Then $\pi$ is scattered if and only if the chains in
its Zhelobenko parameters $\bigcup_{i=1}^l \mathcal{A}_{T_i, t_i} \cup \mathcal{X}_{\mathcal{O}}$ are interlaced.
\end{theorem}
\begin{proof}
There must be a chain $\mathcal{X}_{\mathcal{O}}$ of Type $X$, or else the reduced expressions of $s \in W$ do not contain the short/long root in Type B/C, or one of the roots at the fork in Type $D$.
The fact that they are all interlaced follows directly from the arguments in \cite{DW}.
\end{proof}

Now let us give two applications of Theorem \ref{thm-scattered}.

\subsection{The spin-lowest $K$-type is unitarily small}
The notion of unitarily small $K$-types was introduced by Salamanca-Riba and Vogan \cite{SV} in their uniform conjecture about the unitary dual of real reductive Lie groups.
In our setting, a $K$-type $V_{\mu}^K$ is unitarily small if and only if its highest weight $\mu$ lies in the convex hull generated by the points $\{w\rho\mid w\in W\}$.

\begin{corollary}\label{Cor-usmall}
The spin-lowest $K$-type of any scattered representation is unitarily small.
\end{corollary}
\begin{proof}
Note that by our construction of interlaced chains, the adjacent coordinates of $2\lambda$ differ by at most one.
The corollary therefore follows directly from Lemma 3.4 of \cite{DW}.
\end{proof}

\subsection{Number of scattered representations}
Now let us count the number of
scattered representations as in Section 3 of \cite{DW}.

\begin{corollary}\label{Cor-number}
Let $b_n, c_n, d_n$ be the number of scattered representations of
Type $B$, $C$, $D$ of rank $n$ respectively. Then we have the following recursive
formulas:
\begin{itemize}
\item Type B: $b_2 = 2$,
$$b_{n+1} = \begin{cases}
2b_n-1 & \text{if n is even}\\
2b_n & \text{if n is odd}
\end{cases} \quad \text{for}\ n \geq 2.$$
\item Type C: $c_2 = 3$,
$$c_{n+1} = 2c_n \quad \text{for}\ n \geq 2.$$
\item Type D: $d_3 = 2$, $d_4 = 5$,
$$d_{n+1} = \begin{cases}
2d_n-1 & \text{if n is even}\\
2d_n & \text{if n is odd}
\end{cases} \quad \text{for}\ n \geq 4.$$
\end{itemize}
In particular, the number of scattered representations of Type $C_n$ is given by $c_n = 3\cdot 2^{n-2}$ for all $n \geq 2$.
\end{corollary}

\begin{proof}
Firstly, we list the scattered representations for each type of small rank:
\smallskip

Type $B_2$: $\{3\ 1\}_B$, $\{2\ 1\}_B$.
\smallskip

Type $B_3$: $\{5\ 3\ 1\}_B$, $\{2\} \cup \{3\ 1\}_B$, $\{3\ 2\ 1\}_B$.
\smallskip

Type $C_2$: $\{3\ 1\}_C$, $\{2\}_C \cup \{1\}$, $\{4\ 2\}_C$.
\smallskip

Type $C_3$: $\{5\ 3\ 1\}_C$, $\{2\} \cup \{3\ 1\}_C$, $\{1\} \cup \{4\ 2\}_C$, $\{3\ 1\}\cup \{2\}_C$, $\{6\ 4\ 2\}_C$, $\{3\} \cup \{4\ 2\}_C$.

\smallskip
Type $D_3$: $\{4\ 2\ 0\}_D$, $\{1\} \cup \{2\ 0\}_D$.

\smallskip
Type $D_4$: $\{6\ 4\ 2\ 0\}_D$, $\{3\} \cup \{4\ 2\ 0\}_D$, $\{1\} \cup \{4\ 2\ 0\}_D$, $\{3\ 1\}\cup \{2\ 0\}_D$, $\{3\ 2\ 1\ 0\}_D$.

\medskip

This verifies the corollary for small ranks. For the recursive formula, we can apply an analog of Algorithm 3.6 of \cite{DW} to construct new scattered representations of Type $X_{n+1}$ ($X = B, C, D$) from those of Type $X_n$
(see the example below). Then the result follows.
\end{proof}

\begin{example} \label{eg-lowrank}
We begin with the recursion of Type $C$, which is exactly the same as that in \cite{DW}:
For example, the $c_4 = 12$ scattered representations of Type $C_4$ are obtained from the $c_3 = 6$
scattered representations of Type $C_3$ by:
\begin{align*}
\{5\ 3\ 1\}_C \quad &&\mapsto \quad &&\{7\ 5\ 3\ 1\}_C, &&\{4\} \cup \{5\ 3\ 1\}_C \\
\{2\} \cup \{3\ 1\}_C \quad &&\mapsto \quad &&\{2\} \cup \{5\ 3\ 1\}_C, &&\{4\ 2\} \cup \{3\ 1\}_C \\
\{1\} \cup \{4\ 2\}_C \quad &&\mapsto  \quad &&\{3\} \cup \{1\} \cup \{4\ 2\}_C, &&\{1\} \cup \{6\ 4\ 2\}_C\\
\{3\ 1\} \cup \{2\}_C \quad &&\mapsto \quad &&\{5\ 3\ 1\}\cup \{2\}_C, &&\{3\ 1\}\cup \{4\ 2\}_C \\
\{6\ 4\ 2\}_C \quad &&\mapsto \quad &&\{8\ 6\ 4\ 2\}_C, &&\{5\} \cup \{6\ 4\ 2\}_C \\
\{3\} \cup \{4\ 2\}_C \quad &&\mapsto \quad &&\{5\ 3\} \cup \{4\ 2\}_C, &&\{3\} \cup \{6\ 4\ 2\}_C.
\end{align*}
As for special orthogonal groups, Algorithm 3.6 of \cite{DW} applies as in the
case of symplectic groups {\bf with the following exceptions}:
\begin{itemize}
\item In Type $B_{2n}$, we only have
$$\{2n\ 2n-1\ \dots \ 2\ 1\}_B \mapsto \{2n+1\ 2n\ 2n-1\ \dots \ 2\ 1\}_B.$$
Namely, there is no $\{2n+2\ 2n\ 2n-1\ \dots \ 2\ 1\}_B$.
\item In Type $D_{2n}$, we only have
$$\{2n-1\ 2n-2\ \dots\ \ 1\ 0\}_D \mapsto \{2n\ 2n-1\ 2n-2\ \dots\ 1\ 0\}_D.$$
Namely, there is no $\{2n+1\ 2n-1\ 2n-2\ \dots\ 1\ 0\}_D$.
\end{itemize}
In both cases, the latter parameter does not give a unipotent representation in Theorem \ref{thm-unitarydual}. This explains the
discrepancy between the even and odd $b_n$ and $d_n$ in the formulas of Corollary \ref{Cor-number}.
\end{example}

\section{Spin Lowest $K$-type}

We now investigate the spin lowest $K$-type of all the scattered representations, which must be of the form
\begin{equation} \label{eq-indscattered}
{\rm Ind}_L^G(\pi_L) = {\rm Ind}_{\prod_i GL(a_i) \times G'}^G\left(\bigotimes_i \mathbb{C}_{\chi_i} \otimes \pi_{\mathcal{O}}\right) \in \widehat{G}^{sc},
\end{equation}
according to Theorem \ref{thm-unitarydual}. Here each $\mathbb{C}_{\chi_i}$ has Zhelobenko parameter $\mathcal{A}_i := \mathcal{A}_{T_i, t_i}$. By switching the order of the Levi factors if necessary, we assume that $T_i + t_i \geq T_j + t_j$ for all $i \leq j$.

\medskip

By induction in stages, consider
\begin{equation} \label{eq-pia}
\pi_A := {\rm Ind}_{\prod_i GL(a_i)}^{GL(\sum_i a_i)}(\bigotimes_i \mathbb{C}_{\chi_i}).
\end{equation}
The spin lowest $K$-type of $\pi_A$ is equal to $V^A_{(\theta_1; \dots; \theta_l)}$, where $(\theta_1;\dots;\theta_l)$ is given
in Section 2 of \cite{DW} (our choice of ordering in the previous paragraph guarantees that it is a dominant weight).
Also, the spin lowest $K$-type of $\pi_{\mathcal{O}}$ is given in Sections 5.4 -- 5.6 of \cite{BP1}. More precisely, it is equal to $V_{\theta_{\mathcal{O}}}^{K'}$, where
\begin{equation} \label{eq-thetao}
\theta_{\mathcal{O}} = \begin{cases}
((k+n-1)^2 (k+n-3)^2 \dots (k-n+1)^2 0^{k-n}) & \text{if}\ \mathcal{O} = [2k+1,2n]\ \text{in Type B} \\
(0^{n}) & \text{if}\ \mathcal{O} = [2n]\ \text{in Type C} \\
(n^10^{n-1}) & \text{if}\ \mathcal{O} = [2n-1,1]\ \text{in Type C} \\
(0^{n}) & \text{if}\ \mathcal{O} = [2n]\ \text{in Type D} \\
((n+k-1)^1 (n+k-2)^1\dots (n-k)^1 0^{n-k}) & \text{if}\ \mathcal{O} = [2n,2k-1,1]\ \text{in Type D}
\end{cases},
\end{equation}
where $(n_1^{p_1}, n_2^{p_2}, \dots)$ is the shorthand of $(\overbrace{n_1,\dots,n_1}^{p_1},\overbrace{n_2,\dots,n_2}^{p_2},\dots)$.
In all cases, we separate the non-zero and zero coordinates of $\theta_{\mathcal{O}}$ by writing $\theta_{\mathcal{O}} = (\theta_{\mathcal{O}, +}; 0, \dots, 0)$.

\begin{theorem} \label{thm-slkt}
Let $\pi \in \widehat{G}^{sc}$ be of the form of \eqref{eq-indscattered}. Then the spin lowest $K$-type of $\pi$
is obtained as follows: \\
\noindent \emph{(i)}\ Take $V^A_{(\theta_1;\dots;\theta_l)}$ and $V^{K'}_{\theta_{\mathcal{O}}}$ as given in \eqref{eq-indscattered}.

\noindent\emph{(ii)}\ For each $1 \leq i \leq l$, construct $\mu_i$ from $\theta_i$ by the following:
\begin{itemize}
\item If $\mathcal{A}_i$ is linked with $\mathcal{X}_{\mathcal{O}}$, we have two possibilities:
\begin{itemize}
\item Suppose $\mathcal{A}_i$ and $\mathcal{X}_{\mathcal{O}}$ are linked by
\begin{align*}
\quad \quad \quad \mathcal{A}_i = \{A_1,\ \ \dots,\ \ A_{q-p}&,\ \  A_{q-p+1},\ \ \dots,\ \ A_q \} \\
&\{X_1, \ \ \  \dots\dots,\ \ \ X_p,\ \ \ \ X_{p+1},\ \ \dots, \ \  X_r\} = \mathcal{X}_{\mathcal{O}},
\end{align*}
take $\nu_i := (p,p-1,\dots,1)$, and $\mu_i$ is obtained from $\theta_i$ by adding $\nu_i$ on the
$p$ coordinates $(\dots, \underbrace{(A_1+A_q)/2, \dots, (A_1+A_q)/2}_{p}, \dots)$ of $\theta_i$.

\item Suppose $\mathcal{A}_i$ and $\mathcal{X}_{\mathcal{O}}$ are linked by
\begin{align*}
\mathcal{A}_i = \quad \quad \quad \quad \quad \quad \{A_1, \ \ \dots ,\ \ &A_q\} \\
\{X_1,\ \ \ \ \ \ \dots, \ \ \ \ \ \   X_p,&\ \ \ \ \ \ X_{p+1}, \ \ \dots, \ \ X_r\}  = \mathcal{X}_{\mathcal{O}}
\end{align*}
(this includes the case $\begin{matrix} &\{1\} \\ \{2n\ \dots\ 4 \ 2 \}_C& \end{matrix}$), take $\nu_i := (p,p-1,\dots, p-q+1)$ and
$\mu_i$ is obtained from $\theta_i$ by adding $\nu_i$ on
$\theta_i = (\underbrace{(A_1+A_q)/2,\dots,(A_1+A_q)/2}_{q})$.
\end{itemize}

\item If $\mathcal{A}_i$ is not linked with $\mathcal{X}_{\mathcal{O}}$, take $\mu_i = \theta_i$.
\end{itemize}
\noindent\emph{(iii)}\ Suppose $\mathcal{A}_{i_1}, \dotsm \mathcal{A}_{i_j}$ are the chains that are linked to $\mathcal{X}_{\mathcal{O}}$ such that $i_1 < i_2 < \cdots < i_j$. Let
$\mu_{\mathcal{O}} = (\theta_{\mathcal{O}, +}; \nu_{i_j};\dots;$ $\nu_{i_1}; 0, \dots, 0)$.
\smallskip

Then the spin lowest $K$-type of $\pi$ is given by $V^K_{\mu}$, where $\mu := (\mu_1, \dots, \mu_l;\mu_{\mathcal{O}})$.
\end{theorem}

\begin{example}\label{exam-scattered}
Consider the scattered representation
$${\rm Ind}_{GL(4) \times GL(1) \times GL(2) \times SO(17)}^G(\mathbb{C}_{(15,15,15,15)} \otimes \mathbb{C}_{(8)} \otimes \mathbb{C}_{(5,5)} \otimes \pi_{[15,2]})$$
with Zhelobenko parameter:
$$\begin{aligned}
\{18 && && 16 && && 14 && && 12\} && &&  && && \{8\} && &&\{ 6 && && 4\} && && \\
&& && && && && \{13 && && 11 && && 9 && && 7 && && 5 && && 3 && && 2 && && 1\}_B
\end{aligned}.$$

By our ordering of $\mathcal{A}_i$, we label the chains corresponding to $GL(4)$, $GL(1)$ and $GL(2)$ by $\mathcal{A}_1$, $\mathcal{A}_2$ and $\mathcal{A}_3$, respectively.
By \cite{DW}, the spin lowest $K$-type of $GL$ part is equal to its lowest $K$-type, which is $V^A_{(\theta_1;\theta_2;\theta_3)} = V^A_{(15,15,15,15;8;5,5)}$, and the unipotent
representation $\pi_{[15,2]}$ has spin lowest $K$-type $V^{K'}_{(7,7,0,0,0,0,0,0)}$.
Then  $\mathcal{A}_{1}$, $\mathcal{A}_{2}$ and $\mathcal{A}_{3}$ are all linked to $\mathcal{X}_{\mathcal{O}}$, with
$\nu_{1} = (1)$ on the $GL(4)$ coordinates, $\nu_{2} = (3)$ on the $GL(1)$ coordinates, and $\nu_3 = (5,4)$ on the $GL(2)$ coordinates.
So $\mu = ({\bf 16},15,15,15;{\bf 11};{\bf 10, 9}; 7,7, {\bf 5,4,3,1},0,0)$.

\smallskip
Note that $\{\mu - \rho\} = (\frac{7}{2}, \frac{5}{2}, \frac{3}{2}, \frac{3}{2}, \frac{3}{2}, \frac{3}{2}, \frac{1}{2}, \frac{1}{2}, \frac{1}{2},\frac{1}{2},\frac{1}{2},\frac{1}{2},\frac{1}{2},\frac{1}{2}, \frac{1}{2})$ and hence
$$\{\mu - \rho\} + \rho = \{\mu - \rho\} + (\frac{29}{2}, \frac{19}{2}, \dots, \frac{3}{2}, \frac{1}{2}) = (18,16,14,13,12,11,9,8,7,6,5,4,3,2,1) = 2\lambda$$
satisfies \eqref{eq-par}.
\end{example}

\begin{example}
Let $G = Sp(6,\mathbb{C})$, then the six scattered representations in Example \ref{eg-lowrank} along with their
spin lowest $K$-types are given by

\begin{center}
\begin{tabular}{|c|c|c|c|}
\hline
{\rm Parameters }& {\rm Scattered Representations} & {\rm LKT} & {\rm Spin LKT}\tabularnewline
\hline
$\{5\ 3\ 1\}_{C}$ & $\pi_{[5,1]}$ & $(1,0,0)$ & $(3,0,0)$\tabularnewline
\hline
$\{6\ 4\ 2\}_{C}$ & $\pi_{[6]}$ & $(0,0,0)$ & $(0,0,0)$\tabularnewline
\hline
$\{1\}\cup\{4\ 2\}_{C}$ & ${\rm Ind}_{GL(1)\times Sp(4)}^{G}(\mathbb{\mathbb{C}}_{(1)}\otimes\pi_{[4]})$ & $(1,0,0)$ & $(3,2,0)$\tabularnewline
\hline
$\{2\}\cup\{3\ 1\}_{C}$ & ${\rm Ind}_{GL(1)\times Sp(4)}^{G}(\mathbb{\mathbb{C}}_{(2)}\otimes\pi_{[3,1]})$ & $(2,0,0)$ & $(3,2,1)$\tabularnewline
\hline
$\{3\}\cup\{4\ 2\}_{C}$ & ${\rm Ind}_{GL(1)\times Sp(4)}^{G}(\mathbb{\mathbb{C}}_{(3)}\otimes\pi_{[4]})$ & $(3,0,0)$ & $(4,1,0)$\tabularnewline
\hline
$\{3\ 1\}\cup\{2\}_{C}$ & ${\rm Ind}_{GL(2)\times Sp(2)}^{G}(\mathbb{\mathbb{C}}_{(2,2)}\otimes\pi_{[2]})$ & $(2,2,0)$ & $(3,2,1)$\tabularnewline
\hline
\end{tabular}
\end{center}
For instance, the scattered representation with parameter $\{2\}\cup\{3\ 1\}_{C}$ has
$\theta_1 = (2)$ and $\theta_{\mathcal{O}} = (2,0)$. Here $\nu_1 = (1)$, so
$\mu = (2+1;2,0+1) = (3,2,1)$
\end{example}

In order to prove Theorem \ref{thm-slkt}, it suffices to show that $V^K_{\mu}$ appears in $\pi$, and that $\mu$ satisfies \eqref{eq-par}. Indeed, by the main result of \cite{BDW}, then $V^K_{\mu}$ would be the unique spin-lowest $K$-type
in $\pi$ appearing with multiplicity one. The rest of the manuscript is devoted to proving these two results.

\begin{proposition} \label{prop-1}
Let $\pi$ be a scattered representation of the form given in \eqref{eq-indscattered}. Then
$[\pi|_K: V_{\mu}^K] > 0$.
\end{proposition}
\begin{proof}
Note that we have an inclusion of $M \cap K$-types
$$V_{(\theta_1;\dots;\theta_l)}^A \boxtimes V_{\theta_{\mathcal{O}}}^{K'} \subseteq (\pi_A \boxtimes \pi_{\mathcal{O}})|_{M \cap K},$$
where 
$$M = GL(\sum_i a_i) \times G'$$ 
is a maximal Levi subgroup of $G$ containing the maximal compact torus $T \leq K$, and $\pi_A$ is as defined in \eqref{eq-pia}. Therefore
we have the inclusion of $K$-types
$${\rm Ind}_{Q \cap K}^K(V_{(\theta_1;\dots;\theta_l)}^A \boxtimes V_{\theta_{\mathcal{O}}}^{K'}) \subseteq {\rm Ind}_{Q \cap K}^K((\pi_A \boxtimes \pi_{\mathcal{O}})|_{M \cap K}) \cong \pi|_K,$$
where $Q = MN$ can be chosen as the parabolic subgroup such that the roots of $\mathfrak{n}$
are all contained in $\Delta^+_G$. So to prove the proposition, it suffices to check that
$$
[{\rm Ind}_{Q \cap K}^K(V_{(\theta_1;\dots;\theta_l)}^A \boxtimes V_{\theta_{\mathcal{O}}}^{K'}):V^K_{\mu}] > 0.
$$
This follows immediately
from \eqref{eq-blattner}, Lemmas \ref{lem-1} and \ref{lem-2}.
\end{proof}

We put
$$
W^{\prime}=\{w \in W\mid\ \langle w\rho, \alpha^{\vee} \rangle > 0, \ \forall \alpha \in \Delta_M^+\}.
$$
For induced representations, we have a Blattner-type formula
\begin{equation} \label{eq-blattner}
%\begin{aligned}
[{\rm Ind}_{Q \cap K}^K(V_{(\theta_1;\dots;\theta_l)}^A \boxtimes V_{\theta_{\mathcal{O}}}^{K'}):  V^K_{\mu}] = \sum_{m \in \mathbb{N}}\sum_i (-1)^i [(V_{(\theta_1;\dots;\theta_l)}^A \boxtimes V_{\theta_{\mathcal{O}}}^{K'}) \otimes S^m(\mathfrak{n} \cap \mathfrak{k}): H^i(\mathfrak{n} \cap \mathfrak{k},V^K_{\mu})],
%\end{aligned}
\end{equation}
where
\begin{equation} \label{eq-kostant}
H^i(\mathfrak{n} \cap \mathfrak{k},V^K_{\mu}) = \bigoplus_{\{w \in W^\prime \mid l(w)=i\}} V^{M\cap K}_{w(\mu + \rho) - \rho}.
\end{equation}
The Lie algebra cohomology formula \eqref{eq-kostant} is due to Kostant \cite{K}.
The formula  \eqref{eq-blattner} can be seen by looking at the Weyl character formula of the restricted representation $V^K_{\mu}|_{M\cap K}$ (c.f. \cite{B}, Section 2.7). An analogous formula in the setting of cohomological induction is given in Theorem 5.64 of \cite{KnV}. We note that for complex Lie groups, parabolic induction and cohomological induction are essentially the same.

We now study right hand side of \eqref{eq-blattner} for all $i \geq 0$.
\begin{lemma} \label{lem-1}
There exists a non-negative integer $k$ such that
$$[(V_{(\theta_1;\dots;\theta_l)}^A \boxtimes V_{\theta_{\mathcal{O}}}^{K'}) \otimes S^k(\mathfrak{n} \cap \mathfrak{k}): H^0(\mathfrak{n} \cap \mathfrak{k},V^K_{\mu})] > 0$$
\end{lemma}
\begin{proof}
By \eqref{eq-kostant}, it is obvious that $H^0(\mathfrak{n} \cap \mathfrak{k},V^K_{\mu}) = V^{M \cap K}_{\mu}$. So it suffices to show
$$[(V_{(\theta_1;\dots;\theta_l)}^A \boxtimes V_{\theta_{\mathcal{O}}}^{K'}) \otimes S(\mathfrak{n} \cap \mathfrak{k}): V^{M \cap K}_{\mu}] >0.$$
By our choice of the parabolic subgroup $Q=MN$ in Proposition \ref{prop-1}, we have the following decomposition of $\mathfrak{n} \cap \mathfrak{k}$ as $M \cap K$-modules:
\begin{equation} \label{eq-mfu}
\begin{aligned}
\text{Type B, D:}&\quad \mathfrak{n} \cap \mathfrak{k} = \left(V^A_{(1,0,\dots,0)}  \boxtimes V^{K'}_{(1, 0,\dots,0)}\right) \oplus \left(V^A_{(1,1,0,\dots,0)} \boxtimes V^{K'}_{(0,\dots,0)}\right)\\
\text{Type C:}&\quad \mathfrak{n} \cap \mathfrak{k} = \left(V^A_{(1,0,\dots,0)} \boxtimes V^{K'}_{(1,0,\dots,0)}\right) \oplus \left(V^A_{(2,0,\dots,0)} \boxtimes V^{K'}_{(0,\dots,0)}\right) 
\end{aligned}
\end{equation}
Thus $S^k(\mathfrak{n} \cap \mathfrak{k})$ contains a copy of $S^k(V^A_{(1,0,\dots,0)} \boxtimes V^{K'}_{(1,0,\dots,0)}) = S^k(\mathbb{C}^{\sum_i a_i} \boxtimes \mathbb{C}^{x})$ for each classical type with $x =2\cdot \mathrm{rank}(G')$ for Type C, D, and $x =2\cdot \mathrm{rank}(G') +1$ for Type B.
Note that the $GL(\sum_i a_i) \times GL(x)$-module $S^k(\mathbb{C}^{\sum_i a_i}  \boxtimes \mathbb{C}^{x})$
contains a copy of
$$V^A_{(k_1,\dots,k_j,0,\dots,0)} \boxtimes V^{U(x)}_{(k_1,\dots,k_j,0,\dots,0)}, \quad \text{where}\ k_1 \geq \dots \geq k_j \geq 0\ \text{and}\ k = \sum_i k_i.$$
with multiplicity one. Suppose $j \leq \min \{ \sum_i a_i, \mathrm{rank}(G')\}$. By restricting the $U(x)$-module $V^{U(x)}_{(k_1,\dots,k_j,0,\dots,0)}$ to ${K'}$, it must have a copy of
$V^{K'}_{(k_1,\dots,k_j,0,\dots,0)}$ with multiplicity one.

For any $v \in \mathbb{R}^n$, let $|v|$ be the sum of the coordinates of $v$. By taking $k = |\nu_{i_j}| + \dots + |\nu_{i_1}|$ with each $\nu_{i_t}$ being given by Theorem \ref{thm-slkt}, we conclude that $S^k(\mathfrak{n} \cap \mathfrak{k})$ contains a copy of
$V^A_{(\nu_{i_j},\dots,\nu_{i_1},0,\dots,0)} \boxtimes V^{K'}_{(\nu_{i_j},\dots,\nu_{i_1},0,\dots,0)}.$

On the $GL(\sum_i a_i)$ factor, we will show in Lemma \ref{lemma-LR} that
\begin{equation}\label{eq-branching-typeA}
[V_{(\theta_1;\dots;\theta_l)}^A  \otimes V^A_{(\nu_{i_j},\dots,\nu_{i_1},0,\dots,0)}: V^{A}_{(\mu_1;\dots;\mu_l)}] \geq 1.
\end{equation}
On the $G'$ factor, by using  Section 2.1 of \cite{HTW}, we have that
\begin{equation}\label{eq-branching-Gprime}
[V_{\theta_{\mathcal{O}}}^{K'} \otimes V^{K'}_{(\nu_{i_j},\dots,\nu_{i_1},0,\dots,0)}: V_{\mu_{\mathcal{O}}}^{K'}]\geq 1.
\end{equation}
Indeed,  the number of non-zero entries of  $\mu_{\mathcal{O}} = (\theta_{\mathcal{O}, +}; \nu_{i_j};\dots;$ $\nu_{i_1}; 0, \dots, 0)$ is upper bounded by $\mbox{rank}(G')$. Thus Sections 2.1.2 and 2.1.3 of \cite{HTW} apply, and give the following lower bound for the left hand side of \eqref{eq-branching-Gprime}:
$$
c_{\theta_{\mathcal{O}, \, (\nu_{i_j};\dots; \nu_{i_1}; 0, \dots, 0)}}^{\mu_{\mathcal{O}}}.
$$
One sees that this Littlewood-Richardson coefficient equals one, and \eqref{eq-branching-Gprime} follows.

To summarize, we have
\begin{align*}
&[(V_{(\theta_1;\dots;\theta_l)}^A \boxtimes V_{\theta_{\mathcal{O}}}^{K'}) \otimes S^k(\mathfrak{n} \cap \mathfrak{k}): H^0(\mathfrak{n} \cap \mathfrak{k},V^K_{\mu})] \\
\geq &[(V_{(\theta_1;\dots;\theta_l)}^A \boxtimes V_{\theta_{\mathcal{O}}}^{K'}) \otimes(V^A_{(\nu_{i_j},\dots,\nu_{i_1},0,\dots,0)} \boxtimes V^{K'}_{(\nu_{i_j},\dots,\nu_{i_1},0,\dots,0)}): V^{M \cap K}_{\mu}] \\
 \geq &[V_{(\mu_1;\dots;\mu_l)}^A \boxtimes (V_{\theta_{\mathcal{O}}}^{K'} \otimes V^{K'}_{(\nu_{i_j},\dots,\nu_{i_1},0,\dots,0)}): V^{M \cap K}_{\mu}] \\
 \geq &[V_{(\mu_1;\dots;\mu_l)}^A \boxtimes V_{\mu_{\mathcal{O}}}^{K'}: V^{M \cap K}_{\mu}] =  1,
\end{align*}
where the second inequality uses \eqref{eq-branching-typeA}, while the third inequality uses \eqref{eq-branching-Gprime}. Hence the result follows.
\end{proof}

\begin{remark} \label{rmk-gamma}
From \eqref{eq-mfu} and the proof of the above lemma, one can also check that any $V^A_{\gamma_1} \boxtimes V^{K'}_{\gamma_2}$ appearing
in $S^m(\mathfrak{n} \cap \mathfrak{k})$ must have $|\gamma_1| \geq |\gamma_2|$ for all $m \geq 0$.
\end{remark}

\begin{lemma}\label{lemma-LR}
In the setting of Lemma \ref{lem-1}, the inequality \eqref{eq-branching-typeA} holds.
\end{lemma}
\begin{proof}
By the Littlewood-Richardson Rule as stated on page 420 of \cite{GW},
it suffices to find \emph{one} \emph{L-R skew tableau} of shape $(\mu_1;\dots;\mu_l)/(\theta_1;\dots;\theta_l)$ and weight $(\nu_{i_j},\dots,\nu_{i_1},0,\dots,0)$ in the sense of Definition 9.3.17 of \cite{GW}.
%Let us collect the coordinates of $(\nu_{i_1},\nu_{i_2},\dots,\nu_{i_j})$ from left to right as
%$a_1, a_2, \dots, a_s$. Due to the construction in Theorem \ref{thm-slkt}, we have that
%$$
%a_1> a_2 > \cdots >a_s >0,
%$$
%where $s\geq j$.

Construct the Ferrers diagram $(\mu_1;\dots;\mu_l)/(\theta_1;\dots;\theta_l)$.
Counting from top to bottom, its row sizes are equal to $(\nu_{i_1}, \nu_{i_2}, \dots, \nu_{i_j})$.

We now fill each partition $\nu_{i_t}$ in $(\mu_1;\dots;\mu_l)/(\theta_1;\dots;\theta_l)$ for each $1 \leq t \leq j$
as follows: Let $T$ be the semi-standard Young tableau whose shape is given by the partition
$(\nu_{i_j},\dots,\nu_{i_2},\nu_{i_1})$, and the entries of the $k$-th row of $T$
are all equal to $k$. Take a sequence of sub-tableaux of $T$
$$T_1 \subset T_2 \subset \dots \subset T_j = T$$
such that $T_t$ has the same shape as $(\nu_{i_1}, \nu_{i_2}, \dots, \nu_{i_t})$ for all $1 \leq t \leq j$.
We now look at the skew-tableau $T_t/T_{t-1}$ (where $T_{0}$
is the empty tableau). By construction, the column sizes of
$T_t/T_{t-1}$ are the same as those of $\nu_{i_t}$.

Fill the $k$-th row of the partition $\nu_{i_t}$ in $(\mu_1;\dots;\mu_l)/(\theta_1;\dots;\theta_l)$
by the $k$-th entries on the columns of $T_t/T_{t-1}$ counting from the top in ascending
order. Due to the construction in Theorem  \ref{thm-slkt}, this will give us a \emph{semi-standard skew tableau} of shape $(\mu_1;\dots;\mu_l)/(\theta_1;\dots;\theta_l)$ and weight $(\nu_{i_j}, \dots, \nu_{i_1}, 0, \dots, 0)$ (see Definition 9.3.16 of \cite{GW}), which is a \emph{reverse lattice word} by Definition 9.3.17 of \cite{GW}. Therefore, it is a desired L-R skew tableaux and the proof is complete. \end{proof}
%Let
%\begin{align*}
%w_1 &=\textbf{1}^{a_s},\\
%w_2 &=\textbf{1}^{a_{s-1}-a_s}\textbf{2}^{a_s},\\
%\cdots\\
%w_s&=\textbf{1}^{a_1-a_2} \cdots \textbf{(s-1)}^{a_{s-1}-a_s}\textbf{s}^{a_s},
%\end{align*}
%where $\textbf{1}^{a_s}$ stands for $s$-copies of $1$, and so on.
%Due to the construction in Theorem  \ref{thm-slkt}, filling $w_k$ into the $k$-th row of the Ferrers diagram for $1\leq k\leq s$ gives us a \emph{semi-standard skew tableau} $T$ of shape $(\mu_1;\dots;\mu_l)/(\theta_1;\dots;\theta_l)$ and weight $(\nu_{i_1}, \dots, \nu_{i_j}, 0, \dots, 0)$, see Definition 9.3.16 of \cite{GW}.
 %Moreover, $T$ has row word $w_s\cdots w_2w_1$, which is a \emph{reverse lattice word} by Definition 9.3.17 of \cite{GW}. Therefore, $T$ is a desired L-R skew tableaux and the proof finishes. \end{proof}
\begin{example}\label{exam-LR}
Let us come back to Example \ref{exam-scattered}, where $l=3$ and
$$
(\mu_1;\mu_2;\mu_3)=({\bf 16}, 15,15,15; {\bf 11}; {\bf 10,9}), \quad (\theta_1;\theta_2;\theta_3)=(15,15,15,15;8;5,5).
$$
Recall that $\nu_{i_1}=(1)$, $\nu_{i_2}=(3)$, $\nu_{i_3} = (5,4)$. So the skew Ferrers diagram
$(\mu_1;\mu_2;\mu_3)/(\theta_1;\theta_2;\theta_3)$ looks like:
$$
\begin{tabular}{|c|c|c|c|c|c|c|c|c|c|c|c|c|c|c|c|c|c|}
\hline
$\bullet$ & $\bullet$ & $\bullet$ & $\bullet$ & $\bullet$ & $\bullet$ & $\bullet$ & $\bullet$ &  $\bullet$ & $\bullet$ & $\bullet$ & $\bullet$ & $\bullet$ & $\bullet$ & $\bullet$ &  $\nu_{i_1}$ \tabularnewline
\cline{1-16}
$\bullet$ &$\bullet$ &$\bullet$ &$\bullet$ &$\bullet$ & $\bullet$ & $\bullet$ &  $\bullet$ & $\bullet$ & $\bullet$ & $\bullet$ & $\bullet$ & $\bullet$ & $\bullet$ & $\bullet$ \tabularnewline
\cline{1-15}
$\bullet$ &$\bullet$ &$\bullet$ &$\bullet$ &$\bullet$ & $\bullet$ & $\bullet$ &  $\bullet$ & $\bullet$ & $\bullet$ & $\bullet$ & $\bullet$ & $\bullet$ & $\bullet$ & $\bullet$ \tabularnewline
\cline{1-15}
$\bullet$ &$\bullet$ &$\bullet$ &$\bullet$ &$\bullet$ & $\bullet$ & $\bullet$ &  $\bullet$ & $\bullet$ & $\bullet$ & $\bullet$ & $\bullet$ & $\bullet$ & $\bullet$ & $\bullet$ \tabularnewline
\cline{1-15}
$\bullet$ & $\bullet$ & $\bullet$ & $\bullet$ &  $\bullet$ & $\bullet$ & $\bullet$ & $\bullet$ & $\nu_{i_2}$ & $\nu_{i_2}$ & $\nu_{i_2}$ \tabularnewline
\cline{1-11}
$\bullet$ & $\bullet$ &  $\bullet$ & $\bullet$ & $\bullet$ & $\nu_{i_3}$ & $\nu_{i_3}$ &$\nu_{i_3}$ & $\nu_{i_3}$ & $\nu_{i_3}$ \tabularnewline
\cline{1-10}
$\bullet$ & $\bullet$ &  $\bullet$ & $\bullet$ & $\bullet$ & $\nu_{i_3}$ & $\nu_{i_3}$ &$\nu_{i_3}$ & $\nu_{i_3}$  \tabularnewline
\cline{1-9}
\end{tabular}
$$
To fill in the entries of the above diagram, consider
$$T_1 = \begin{tabular}{|c|c}
\cline{1-1}
\bf{1} & \tabularnewline
\cline{1-1}
\multicolumn{1}{c}{} & \tabularnewline
\end{tabular} \subset \quad
T_2 = \begin{tabular}{|c|c|c|c|c|}
\hline
$1$ & ${\bf 1}$ & ${\bf 1}$   \tabularnewline
\cline{1-3}
${\bf 2}$    \tabularnewline
\cline{1-1}
\end{tabular} \quad
\subset \quad
\begin{tabular}{|c|c|c|c|c|}
\hline
$1$ & $1$ & $1$ & ${\bf 1}$ & ${\bf 1}$  \tabularnewline
\cline{1-5}
$2$ & ${\bf 2}$ & ${\bf 2}$ & ${\bf 2}$   \tabularnewline
\cline{1-4}
${\bf 3}$ & ${\bf 3}$ &  ${\bf 3}$  \tabularnewline
\cline{1-3}
${\bf 4}$  \tabularnewline
\cline{1-1}
\end{tabular} = T_3 = T,$$
where the highlighted blocks are $T_t/T_{t-1}$ for $t = 1,2$. This leads us to the following tableau:
%Adopting the setting of Lemma \ref{lemma-LR}, we get $s=2$ and $a_1=2> a_2=1$. Then $w_1=1$ and $w_2=12$.  Form the skew Ferrers diagram $(\mu_1;\mu_2;\mu_3)/(\theta_1;\theta_2;\theta_3)$. Fill $w_1$ to its first row, and fill $w_2$ to its second row.
$$
\begin{tabular}{|c|c|c|c|c|c|c|c|c|c|c|c|c|c|c|c|c|c|}
\hline
$\bullet$ & $\bullet$ & $\bullet$ & $\bullet$ & $\bullet$ & $\bullet$ & $\bullet$ & $\bullet$ &  $\bullet$ & $\bullet$ & $\bullet$ & $\bullet$ & $\bullet$ & $\bullet$ & $\bullet$ &  $1$ \tabularnewline
\cline{1-16}
$\bullet$ &$\bullet$ &$\bullet$ &$\bullet$ &$\bullet$ & $\bullet$ & $\bullet$ &  $\bullet$ & $\bullet$ & $\bullet$ & $\bullet$ & $\bullet$ & $\bullet$ & $\bullet$ & $\bullet$ \tabularnewline
\cline{1-15}
$\bullet$ &$\bullet$ &$\bullet$ &$\bullet$ &$\bullet$ & $\bullet$ & $\bullet$ &  $\bullet$ & $\bullet$ & $\bullet$ & $\bullet$ & $\bullet$ & $\bullet$ & $\bullet$ & $\bullet$ \tabularnewline
\cline{1-15}
$\bullet$ &$\bullet$ &$\bullet$ &$\bullet$ &$\bullet$ & $\bullet$ & $\bullet$ &  $\bullet$ & $\bullet$ & $\bullet$ & $\bullet$ & $\bullet$ & $\bullet$ & $\bullet$ & $\bullet$ \tabularnewline
\cline{1-15}
$\bullet$ & $\bullet$ & $\bullet$ & $\bullet$ &  $\bullet$ & $\bullet$ & $\bullet$ & $\bullet$ & $1$ & $1$ & $2$ \tabularnewline
\cline{1-11}
$\bullet$ & $\bullet$ &  $\bullet$ & $\bullet$ & $\bullet$ & $1$ & $1$ &$2$ & $2$ & $3$ \tabularnewline
\cline{1-10}
$\bullet$ & $\bullet$ &  $\bullet$ & $\bullet$ & $\bullet$ & $2$ & $3$ &$3$ & $4$  \tabularnewline
\cline{1-9}
\end{tabular}
$$
Note that the row word of $T$ is $2334112231121$, which is a reverse lattice word.  Thus $T$ is an L-R skew tableau of shape $(\mu_1;\mu_2;\mu_3)/(\theta_1;\theta_2;\theta_3)$ and weight $(5,4,3, 1, 0, 0, 0, 0)$. Now the Littlewood-Richardson Rule guarantees that
$$
[V_{(15, 15,15,15; 8; 5,5)}^A  \otimes V^A_{(5,4,3,1,0,0,0)}: V^{A}_{({\bf 16}, 15,15,15; {\bf 11}; {\bf 10, 9})}] \geq 1.
$$
\end{example}

\begin{lemma} \label{lem-2}
For all $i > 0$ and $m \geq 0$, we have
\begin{equation} \label{eq-vanish}
[(V_{(\theta_1;\dots;\theta_l)}^A \boxtimes V_{\theta_{\mathcal{O}}}^{K'}) \otimes S^m(\mathfrak{n} \cap \mathfrak{k}): H^i(\mathfrak{n} \cap \mathfrak{k},V^K_{\mu})] = 0.
\end{equation}
\end{lemma}
\begin{proof}
Let $\theta_A := (\theta_1;\dots;\theta_l)$. By Remark \ref{rmk-gamma}, the $M \cap K$-types appearing
on the left hand side of \eqref{eq-vanish} must be of the form
$$V_{\theta_A + \gamma_1}^A \boxtimes V_{\theta_{\mathcal{O}}+\gamma_2}^{K'}, \quad \quad |\gamma_1| \geq |\gamma_2|,$$
where $\gamma_1$ and $\gamma_2$ consist solely of non-negative integers.
We \emph{claim} that for all $i > 0$, if $H^i(\mathfrak{n} \cap \mathfrak{k},V^K_{\mu})$ consist of $M \cap K$-types of the form
$$V_{\theta_A + \delta_1}^A \boxtimes V_{\theta_{\mathcal{O}}+\delta_2}^{K'}$$
where $\delta_1, \delta_2$ consists only of non-negative integers, then $|\delta_1| < |\delta_2|$.

Indeed, recall from the construction of $\mu$ in Theorem \ref{thm-slkt} that
$$\mu = (\overbrace{\quad \theta_A + \omega_1 \quad}^{q := \sum_i a_i}; \overbrace{\quad \theta_{\mathcal{O}} + \omega_2 \quad}^{r := {\rm rank}(G')} ), \quad \quad |\omega_1| = |\omega_2|.$$
If $w \in W$ is so that $w(\mu + \rho) = (\beta_1; \beta_2)$ is regular on $\Delta_L^+$, $\beta_1$, $\beta_2$ are obtained by the following:
\begin{enumerate}
\item[(a)] Take any sub-collection of $q$ entries inside $\mu + \rho$, and assign either $+$ or $-$ to each coordinate;
\item[(b)] $\beta_1$ is obtained by rearranging the $q$ coordinates chosen in (a) in descending order;
\item[(c)] For the remaining $r$ entries of $\mu + \rho$, rearrange in descending order and get $\beta_2$.
\end{enumerate}
It is obvious that if $w(\mu + \rho) - \rho = (\beta_1;\beta_2) - \rho$ contributes to any multiplicities in \eqref{eq-vanish},
the entries of $\beta_1$ must be all positive, i.e., we always assign $+$ in Step (a) above. So we focus on
$w \in W$ consisting of transpositions only, which implies that $|w(\mu + \rho) - \rho| = |\mu|$.

Therefore, if $w$ is not the identity, then
the sum of the first $q$ coordinates of $w(\mu+\rho)-\rho$ must be strictly less than that of $\theta_A + \omega_1$,
and the sum of the last $r$ coordinates of $w(\mu+\rho)-\rho$ must be strictly greater than that of $\theta_{\mathcal{O}} + \omega_2$. This proves our claim, and the lemma follows immediately.
\end{proof}

The proof of Theorem \ref{thm-slkt} ends with:
\begin{proposition}
Equation \eqref{eq-par} holds for $V^K_{\mu}$, i.e., $\{\mu - \rho\} = 2\lambda - \rho.$
\end{proposition}
\begin{proof}
In our construction of $V^K_{\mu} = V^K_{(\mu_1;\dots;\mu_l;\mu_{\mathcal{O}})}$,
the coordinates of Type $A$ chains $\mu_i$ are determined in exactly the same way
as in Algorithm 2.2 of \cite{DW} (this is true also when it is linked to $\mathcal{X}_{\mathcal{O}}$).
Hence the proof in \cite{DW} applies to all $\mu_i$ appearing in $\mu$.

We now focus on studying the coordinates corresponding to $\theta_{\mathcal{O}} \mapsto \mu_{\mathcal{O}}$. For convenience, we reorder
$\mathcal{A}_i$ (if necessary) such that
\begin{equation} \label{eq:interlaced}
\begin{aligned}
\{\quad \quad \quad &\ \mathcal{A}_{1}\quad \quad \quad \}   &&\{\overbrace{\ \ \mathcal{A}_{2}\ \ }^{q_2}\} \quad \cdots\cdots && \{\overbrace{\ \ \ \mathcal{A}_{k}\ \ \ }^{q_k}\}\\
\{&\underbrace{X_{1}, \cdots, X_p}_{p},\ \ \cdots\cdots  &&\cdots \cdots\ \cdots\cdots\ \ \cdots\cdots &&\cdots\cdots ,\ \underbrace{X_{r-Z+1}, \cdots,\ X_r}_{Z}\}_X = \mathcal{X}_{\mathcal{O}}.
\end{aligned}
\end{equation}
In particular, we only study the last $r = \mathrm{rank}(G')$ coordinates of $\{\mu - \rho\}$ and $2\lambda - \rho$ in our calculations thereafter.

\medskip
Let $\lambda_{\mathcal{O}}$ be such that $2\lambda_{\mathcal{O}}$ is equal to the coordinates of $\mathcal{X}_{\mathcal{O}}$.
Since $V^{K'}_{\theta_{\mathcal{O}}}$ is the spin lowest $K$-type of $\pi_{\mathcal{O}}$ by Sections 5.4 -- 5.6 of \cite{BP1}, we have
\begin{equation} \label{eq-xo}
2\lambda_{\mathcal{O}} - \rho_{r} = \{\theta_{\mathcal{O}} - \rho_{r}\},
\end{equation}
where $\rho_k := (h_k, \dots, h_1)$ is half sum of the positive roots in the Dynkin diagram of Lie type $X_k$ $(X = B, C, D)$.

Let 
$$\Lambda := (\text{the last}\ r\ \text{coordinates of}\ 2\lambda - \rho) - (2\lambda_{\mathcal{O}} - \rho_{r}),$$
$$\Theta := (\text{the last}\ r\ \text{coordinates of}\ \{\mu - \rho\}) - \{\theta_{\mathcal{O}} - \rho_{r}\}.$$
If $\Lambda = \Theta$, then we can conclude that $2\lambda - \rho = \{\mu - \rho\}$ by \eqref{eq-xo}, and the proposition follows.
\medskip

We consider $\Lambda$ first. The entries of $2\lambda$ are given by \eqref{eq:interlaced}. Subtracting it by $\rho$, one has 
$$
2\lambda - \rho =\quad \quad \begin{aligned}
\cdots\cdots  &\quad \quad \quad \{\overbrace{\dots\quad , \ast -h_{Z+3},\quad \quad \ast -h_{Z+1}  }^{q_k}\}\\
\{  \cdots, &X_{r-Z-1} -h_{Z+4} ,\quad X_{r-Z} -h_{Z+2},\quad \underbrace{X_{r-Z+1}-h_Z, \cdots,\ X_r-h_1}_{Z}\}_X.
\end{aligned}
$$
Meanwhile,
$$
2\lambda_{\mathcal{O}} - \rho_r = \{X_1 - h_r,  \cdots, X_{r-Z-1} -h_{Z+2}, X_{r-Z} -h_{Z+1},  X_{r-Z+1}-h_Z, \cdots, X_r-h_1 \}_X.
$$
Therefore, we have
\begin{equation} \label{eq-Lambda}
\Lambda = ((\underbrace{\nu_k;\dots;\nu_{1};0,\dots,0}_{r-Z})^t;\underbrace{0,\dots,0}_{Z}),
\end{equation}
where $\nu_i$ are determined in Theorem \ref{thm-slkt} for $1 \leq i \leq k$, and ${\bf p}^t$ is the transpose of the
partition ${\bf p}$ by switching the rows of $p$ into columns. In other words, if ${\bf p} = (\alpha_1, \dots, \alpha_{r-Z})$,
then
$${\bf p}^t := (\beta_1, \dots, \beta_{r-Z}), \quad \quad \text{where}\ \beta_i = \#\{j\ |\ \alpha_j \geq i\} \quad \forall\ i \geq 1.$$
%the difference between the last $r$ coordinates
%of $\{\mu - \rho\}$ and $\{\theta_{\mathcal{O}} - \rho_r\}$ is equal to
%$$(\rho_{r-Z} - \{\rho_{r-Z} - (\underbrace{\nu_1;\nu_2;\dots;\nu_{k+1};0,\dots,0}_{r-Z})\}; \underbrace{0,\dots,0}_{Z}).$$
%
%More precisely,

On the other hand, note that $0 \leq s :=$ number of coordinates of $\theta_{\mathcal{O},+} \leq Z \leq r$, and
\begin{align*}
&\{\theta_{\mathcal{O}} - \rho_r\} \\
=\ &\{(\theta_{\mathcal{O},+}; 0,\dots,0) - (h_r, \dots, h_1)\} \\
=\ &\{(\underbrace{\theta_{\mathcal{O},+} - (h_r, \dots, h_{r-s+1})}_{s}; \underbrace{(0,\dots, 0)- (h_{r-s}, \dots, h_{Z-s+1})}_{r-Z};
\underbrace{(0, \dots, 0)-(h_{Z-s}, \dots, h_{1})}_{Z-s})\} \\
=\ &(\underbrace{h_{r-s}, \dots, h_{Z-s+1}}_{r-Z}; \underbrace{h_{Z-s}, \dots, h_1; \{\theta_{\mathcal{O},+} - (h_r,\dots, h_{r-s+1})\}}_{Z}).
\end{align*}
Note that the coordinates of $\{\theta_{\mathcal{O},+} - (h_1,\dots, h_s)\}$
are all equal to either $\frac{1}{2}$ or $0$ by direct calculation on $\theta_{\mathcal{O},+}$ given by \eqref{eq-thetao}, or by looking at Equations (20), (22) and (24) of \cite{BDW}.
Hence the expression in the last equality above is dominant.

Meanwhile, recall $\mu = (\theta_{\mathcal{O},+};\nu_k; \dots; \nu_{1};0,\dots,0)$. So
the last $r$ coordinates of $\{\mu-\rho\}$ are equal to:
%&\{(\theta_{\mathcal{O},+} - (p_1,\dots,p_s);&(\nu_1; \dots; \nu_{k+1}; 0,\dots,0) - (p_{s+1},\dots,p_{s+r-Z}); \\
 %& &(0,\dots,0) - (p_{r - (Z-s) + 1},\dots,p_{r}))\} \\
$$(\{\underbrace{(\nu_k; \dots; \nu_{1}; 0,\dots,0)-(h_{r-s}, \dots, h_{Z-s+1})}_{r-Z}\}; \underbrace{h_{Z-s}, \dots, h_1; \{\theta_{\mathcal{O},+} - (h_r,\dots, h_{r-s+1})\}}_{Z})$$

Writing $\ell := (h_{r-s}, \dots, h_{Z-s+1})$, the difference of the above equations gives
\begin{equation} \label{eq-Theta}
\Theta = (\ell - \{(\nu_k; \dots; \nu_{1}; 0,\dots,0)-\ell\}; 0,\dots,0).
\end{equation}
By comparing \eqref{eq-Lambda} and \eqref{eq-Theta}, the proposition follows if one can show that
$$\ell - \{(\nu_k; \dots; \nu_{1}; 0,\dots,0)-\ell\} = (\nu_k; \dots; \nu_{1}; 0,\dots,0)^t$$
as elements of $\mathbb{N}^m := \mathbb{N}^{r-Z}$, or equivalently
\begin{equation} \label{eq-ell}
\{\ell - (\nu_k; \dots; \nu_{1}; 0,\dots,0)\} = \ell - (\nu_k; \dots; \nu_{1}; 0,\dots,0)^t.
\end{equation}

The coordinates of $\ell$ in \eqref{eq-ell} can be translated by any fixed integer as long as the
coordinates inside the braces on the left remain non-negative. In particular,
we can prove \eqref{eq-ell} holds by replacing $\ell$ with $\rho_m$.
\medskip

For simplicity, we only prove \eqref{eq-ell} in Type $C$ with $\ell = \rho_m = (m,\dots,2,1)$.
We claim that for all partitions ${\bf p} = (p_1,\dots,p_m)$ such
that $p_1 \leq m$ and all {\it positive} entries of ${\bf p}$ are distinct (for example, ${\bf p} = (\nu_k;\dots;\nu_{1};0,\dots,0)$),
\begin{equation} \label{eq-pt}
\{\ell - {\bf p}\} = \ell - {\bf p}^t.
\end{equation}

\begin{example} \label{eg-pt}
Let $\ell = \rho_{10}$ and ${\bf p} = (10,7,5,4,1)$. Then we have
$$\{\ell - {\bf p}\} = \{10-10,9-7,8-5,7-4,6-1,5,4,3,2,1\} = (5,5,4,3,3,3,2,2,1,0)$$
$$\ell - {\bf p}^t = (10,9,8,7,6,5,4,3,2,1) - (5,4,4,4,3,2,2,1,1,1) = (5,5,4,3,3,3,2,2,1,0).$$
Therefore, \eqref{eq-pt} holds.
\end{example}

We now give a proof of \eqref{eq-pt}. By hypothesis, $\ell / {\bf p}$ defines a skew partition, 
whose row and column sizes give the sizes of $\{\ell - {\bf p}\}$ and $\ell - {\bf p}^t$ respectively. 
So we need to show that $\ell / {\bf p}$ have the same row and column sizes.

Mark the $(i,j)$-block of $\ell / {\bf p}$ by $(m+2) - (i+j)$, so that the leftmost entry of each row of
$\ell / {\bf p}$ gives the size of the row, and the topmost entry of each column of $\ell / {\bf p}$
gives the size of the column. For instance, in the setting of Example \ref{eg-pt}, we have
\begin{center}
\begin{tabular}{|c|c|c|c|c|c|cccc}
\hline
$\bullet$ & $\bullet$ &$\bullet$  & $\bullet$ & $\bullet$ & $\bullet$ & \multicolumn{1}{c|}{$\bullet$} & \multicolumn{1}{c|}{$\bullet$} & \multicolumn{1}{c|}{$\bullet$} & \multicolumn{1}{c|}{$\bullet$}\tabularnewline
\hline
 $\bullet$&$\bullet$  & $\bullet$ & $\bullet$ &  $\bullet$& $\bullet$ & \multicolumn{1}{c|}{$\bullet$} & \multicolumn{1}{c|}{$2$} & \multicolumn{1}{c|}{$1$} & \tabularnewline
\cline{1-9}
 $\bullet$&  $\bullet$& $\bullet$ & $\bullet$ & $\bullet$ & $3$ & \multicolumn{1}{c|}{$2$} & \multicolumn{1}{c|}{$1$} &  & \tabularnewline
\cline{1-8}
 $\bullet$& $\bullet$ & $\bullet$ & $\bullet$ & $3$ & $2$ & \multicolumn{1}{c|}{$1$} &  &  & \tabularnewline
\cline{1-7}
 $\bullet$ & $5$ & $4$ & $3$ & $2$ & $1$ &  &  &  & \tabularnewline
\cline{1-6}
$5$ & $4$ & $3$ & $2$ & $1$ & \multicolumn{1}{c}{} &  &  &  & \tabularnewline
\cline{1-5}
$4$ & $3$ & $2$ & $1$ & \multicolumn{1}{c}{} & \multicolumn{1}{c}{} &  &  &  & \tabularnewline
\cline{1-4}
$3$ & $2$ & $1$ & \multicolumn{1}{c}{} & \multicolumn{1}{c}{} & \multicolumn{1}{c}{} &  &  &  & \tabularnewline
\cline{1-3}
$2$ & $1$ & \multicolumn{1}{c}{} & \multicolumn{1}{c}{} & \multicolumn{1}{c}{} & \multicolumn{1}{c}{} &  &  &  & \tabularnewline
\cline{1-2}
$1$ & \multicolumn{1}{c}{} & \multicolumn{1}{c}{} & \multicolumn{1}{c}{} & \multicolumn{1}{c}{} & \multicolumn{1}{c}{} &  &  &  & \tabularnewline
\cline{1-1}
\end{tabular}
\end{center}
so that its \textit{nonzero} row and column sizes are $\{\ell - {\bf p}\} = \{2,3,3,5,5,4,3,2,1\}$ (counting from top to bottom) and $\ell - {\bf p}^t = (5,5,4,3,3,3,2,2,1)$
(counting from left to right) respectively.

\medskip
We identify the entries of $\{\ell - {\bf p}\}$ with that of $\ell - {\bf p}^t$ as follows:

\medskip
\noindent {\bf (i)}\ If the leftmost block of a row of $\ell / {\bf p}$ is also the topmost block of a column of $\ell / {\bf p}$,
then we have a natural identification between the entries of $\{\ell - {\bf p}\}$ and $\ell - {\bf p}^t$ corresponding to this block.
  
For instance, the blocks satisfying this property in Example \ref{eg-pt} are circled below:
\begin{center}
\begin{tabular}{|c|c|c|c|c|c|cccc}
\hline
$\bullet$ & $\bullet$ &$\bullet$  & $\bullet$ & $\bullet$ & $\bullet$ & \multicolumn{1}{c|}{$\bullet$} & \multicolumn{1}{c|}{$\bullet$} & \multicolumn{1}{c|}{$\bullet$} & \multicolumn{1}{c|}{$\bullet$}\tabularnewline
\hline
 $\bullet$&$\bullet$  & $\bullet$ & $\bullet$ &  $\bullet$& $\bullet$ & \multicolumn{1}{c|}{$\bullet$} & \multicolumn{1}{c|}{\textcircled{$2$}} & \multicolumn{1}{c|}{$1$} & \tabularnewline
\cline{1-9}
 $\bullet$&  $\bullet$& $\bullet$ & $\bullet$ & $\bullet$ & \textcircled{$3$} & \multicolumn{1}{c|}{$2$} & \multicolumn{1}{c|}{$1$} &  & \tabularnewline
\cline{1-8}
 $\bullet$& $\bullet$ & $\bullet$ & $\bullet$ & \textcircled{$3$} & $2$ & \multicolumn{1}{c|}{$1$} &  &  & \tabularnewline
\cline{1-7}
 $\bullet$ & \textcircled{$5$} & $4$ & $3$ & $2$ & $1$ &  &  &  & \tabularnewline
\cline{1-6}
\textcircled{$5$} & $4$ & $3$ & $2$ & $1$ & \multicolumn{1}{c}{} &  &  &  & \tabularnewline
\cline{1-5}
$4$ & $3$ & $2$ & $1$ & \multicolumn{1}{c}{} & \multicolumn{1}{c}{} &  &  &  & \tabularnewline
\cline{1-4}
$3$ & $2$ & $1$ & \multicolumn{1}{c}{} & \multicolumn{1}{c}{} & \multicolumn{1}{c}{} &  &  &  & \tabularnewline
\cline{1-3}
$2$ & $1$ & \multicolumn{1}{c}{} & \multicolumn{1}{c}{} & \multicolumn{1}{c}{} & \multicolumn{1}{c}{} &  &  &  & \tabularnewline
\cline{1-2}
$1$ & \multicolumn{1}{c}{} & \multicolumn{1}{c}{} & \multicolumn{1}{c}{} & \multicolumn{1}{c}{} & \multicolumn{1}{c}{} &  &  &  & \tabularnewline
\cline{1-1}
\end{tabular},
\end{center}
with $\{\ell - {\bf p}\} = \{$\textcircled{2},\textcircled{3},\textcircled{3},\textcircled{5},\textcircled{5},4,3,2,1$\}$, $\ell - {\bf p}^t = ($\textcircled{5},\textcircled{5},4,3,\textcircled{3},\textcircled{3},2,\textcircled{2},1$)$.
\medskip

\noindent {\bf (ii)} Consider the unidentified entries of $\{\ell - {\bf p}\}$ in {\bf (i)}. They correspond to the entries
in the leftmost block of a row in $\ell / {\bf p}$ but not the topmost block of any column. 
Since ${\bf p}$ is a strictly decreasing partition, these blocks must occur at the first column of $\ell / {\bf p}$. More precisely,
if ${\bf p} = (p_1 > \dots > p_t > p_{t+1} = 0 \dots p_m = 0)$, then the entry of $(t+1,1)$-block of $\ell / {\bf p}$
is $m-t$, and the blocks below it cannot be the topmost block of a column. These blocks take the entries:
\begin{equation} \label{eq-a}
\{m-t-1, \dots, 2,1\}
\end{equation}
(in our example, $m = 10$, $t = 5$ and hence $m-t-1 = 4$).

\medskip
\noindent {\bf (iii)} On the other hand, we study the unidentified entries of $\ell - {\bf p}^t$ in {\bf (i)}. They are the entries
at the topmost block of a column but not the leftmost block of any row. 
Namely,
suppose the entries of the $i^{th}$-row of $\ell - {\bf p}$ are $v_i$, $v_i - 1$, $\dots$, $1$,
then the blocks corresponding to $v_i -1$, $v_i -2$, $\dots$, $v_{i-1}$ are the topmost blocks of some columns that
are not the leftmost blocks of row $i$. 

For instance, by looking the fifth row of the skew tableau in {\bf (i)}, the entries on this row
are \textcircled{5}, 4, 3, 2, 1, and the blocks with entries $4$ and $3$ contribute to the 
unidentified entries in $\ell - {\bf p}^t$.

Collecting all such entries on each row of $\ell / {\bf p}$, we have
\begin{equation} \label{eq-b}
\bigcup_{i = 1}^{t+1} \{v_i - 1, v_i -2, \dots, v_{i-1}\} = \{v_{t+1}-1, v_{t+1} -2, \dots, 2, 1\}.
\end{equation}
Note that the entry $v_{t+1}$ appears at the $(t+1,1)$-block, so $v_{t+1} = m-t$ as in {\bf (ii)}.

\medskip
Consequently, \eqref{eq-a} and \eqref{eq-b} are equal, and we have established 
an identification between the entries of $\{\ell - {\bf p}\}$ and $\ell - {\bf p}^t$.
Therefore \eqref{eq-pt} holds, and the result follows.
\end{proof}

\section{On a conjecture of Huang} \label{sec-Hconj}

Let us investigate Conjecture 13.3 of \cite{H} raised by Huang in 2015.

\begin{conjecture}\emph{(\cite{H})}\label{conj-Huang}
A unitary representation either has nonzero Dirac cohomology or is induced from a unitary representation with nonzero Dirac cohomology by parabolic induction.
\end{conjecture}

\begin{example}
Let $G$ be $Sp(6, \bC)$. Fix a positive root system $\Delta_G^+$ so that it has simple roots  $\{e_1-e_2, e_2-e_3, 2e_3\}$. Consider the  spherical irreducible unitary representation
$
J(\lambda, \lambda)
$
with $\lambda=(\frac{5}{2}, \frac{3}{2}, \frac{1}{2})$. This is the metaplectic representation $\pi_{\rm even}$ described in Section 5.5 of \cite{BP1}. As computed there, $H_D(\pi_{\rm even})=0$.

We \emph{claim} that $\pi_{\rm even}$ cannot be parabolically induced from any unitary representation with non-zero Dirac cohomology. Indeed, if there exists such a representation, say $\pi_L$, then
$$
\pi_{\rm even}={\rm Ind}_{L}^G(\pi_L).
$$
Since the infinitesimal character $2\lambda=(5,3,1)$ of $\pi_{\rm even}$ is dominant, integral and regular for $\Delta_G^+$, one would conclude from Theorem 2.4 of \cite{BP1} that $H_D(\pi_{\rm even})\neq 0$, contradiction.  Thus  the claim holds, and $\pi_{\rm even}$ violates Conjecture \ref{conj-Huang}.

More generally, there are other unipotent representations in $G = Sp(2n,\mathbb{C})$ violating the conjecture. Consider the spherical special unipotent representation $\pi_{\mathcal{O},1}$ corresponding to a nonzero, cuspidal special nilpotent orbit $\mathcal{O} \subset \mathfrak{g}$. For example, one can take $\mathcal{O} = [4m_1, 4m_2, \dots, 4m_k]$ with integers $m_1 > m_2 > \dots > m_k > 0$. Then $H_D(\pi_{\mathcal{O},1}) = 0$, since $h^{\vee} = 2\lambda$ in Equation \eqref{eq-par} is singular (here $h^{\vee}$ is the semisimple element of a Jacobson-Morozov triple of the Lusztig-Spaltenstein dual $\mathcal{O}^{\vee} \subset \mathfrak{g}^{\vee}$). On the other hand, $\pi_{\mathcal{O},1}$ cannot be parabolically induced from any representations given in Theorem \ref{thm-unitarydual} tensored with a unitary character of Type $A$, or else its associated variety ($\overline{\mathcal{O}}$ in this case) would be a non-cuspidal nilpotent orbit.\hfill\qed
\end{example}

\medskip
\centerline{\scshape Acknowledgements}
The first named author would like to thank his thesis adviser Prof.~Huang sincerely for explaining Conjecture \ref{conj-Huang} to him during a conference held at the Chern Institute of Mathematics in June 2014. In particular, ``parabolic induction" in Conjecture \ref{conj-Huang} means full parabolic induction. That is, one does not need to pass to composition factors.

We thank the two referees sincerely for their very careful reading and very nice suggestions.

\medskip
\centerline{\scshape Funding}
Dong was supported by the National Natural Science Foundation of China (grant 11571097, 2016-2019). Wong is supported by the National Natural Science Foundation of China (grant 11901491) and the Presidential Fund of CUHK(SZ).

\end{document}